\documentclass[12pt]{amsart}
\usepackage{amstext,amsfonts,amssymb,amscd,amsbsy,amsmath,verbatim,fullpage}
\usepackage[alphabetic,lite,backrefs]{amsrefs} 
\usepackage{ifthen}
\usepackage{color,tikz}
\usepackage{amsthm}
\usepackage{latexsym}
\usepackage[all]{xy}
\usepackage{enumerate}

\newtheorem{lemma}{Lemma}[section]

\newtheorem{claim*}{Claim}
\newtheorem{thm}[lemma]{Theorem}

\theoremstyle{remark}

\newtheorem{example}[lemma]{Example}

\newcommand{\pdim}{\operatorname{pdim}}

\newcommand{\Tor}{\operatorname{Tor}}

\newcommand{\Pic}{\operatorname{Pic}}

\newcommand{\reg}{\operatorname{reg}}

\newcommand{\cO}{{\mathcal O}}

\newcommand{\codim}{\operatorname{codim}}

\newcommand{\defi}[1]{\textsf{#1}} 

\newcommand{\PP}{\mathbb P}

\newcommand{\QQ}{\mathbb Q}

\title{About how large are algebraic Betti numbers?}
\date{\today}
\author{Daniel Erman}

\begin{document}
\maketitle
\begin{abstract}
We use Boij-S\"oderberg theory to provide some order of magnitude bounds on algebraic Betti numbers.
\end{abstract}

\section{Introduction}\label{sec:intro}

Consider $\mathbb P^2\subseteq \PP^{5150}$ given by the degree $100$ Veronese embedding.  Let $S$ be the coordinate ring of $\PP^{5150}$ and let $I\subseteq S$ be the defining ideal of the image of $\PP^2$.  Write $\beta_{i}(S/I)$ for the corresponding algebraic Betti number, i.e. $\beta_{i}(S/I) = \dim \Tor_{i}(S/I,k)$ is the number of generators of the $i$'th syzygy module of $S/I$. About how many digits does $\beta_{i}(S/I)$ have? Our main results will give estimates for this type of question.  For instance, we will show that  in this example, $\beta_{2000}(S/I)$ has between 1484 and 1499 digits.

In~\cite{EL-asymptotic}, Ein and Lazarsfeld coined the phrase \defi{asymptotic syzygies} for the syzygies of algebraic varieties under increasingly positive embeddings.  The first asymptotic syzygy results are perhaps the $N_p$ theorems of Mark Green~\cite{green1}, which showed that for smooth curves, as the positivity of the embedding increased, one obtained linear syzygies for more and more steps of the free resolution; see also~\cite{green-laz2, green-laz1, lazarsfeld-sampling}.  Many later papers like~\cite{EL-higher} extended $N_p$ theorems to higher dimensional varieties, but Ein and Lazarsfeld observed that, for very positive embeddings, these $N_p$ results only describe a very small percentage of Betti diagram, and they began asking asymptotic questions about syzygies~\cite{EL-asymptotic, EL-progress-and-questions}.  

The qualitatitve asymptotic picture (which Betti numbers are nonzero?) is now largely complete~\cite{EL-asymptotic,park,raicu}.  We consider a small part of the quantitative story (how large are the Betti numbers?), which remains more open.

\medskip

We work throughout over an arbitrary field $k$.  We will consider $\PP^n$ under the $d$-uple Veronese embedding.  To simplify later notation, we write $N = \binom{d+n}{n}-n-1$ so we have $\PP^n \subseteq \PP^{N+n}$ and $N$ is the codimension.  We write $S$ for the coordinate ring of $\PP^{N+n}$ and $I\subseteq S$ for the defining ideal of the image of $\PP^n$.  Write $\beta_{i}(\PP^n;d):=\dim_k \Tor_{i}(S/I,k)$ for the $i$'th total Betti number of $S/I$.
Our main result about $\PP^n$ is the following:
\begin{thm}\label{thm:mainPn}
Fix $n,d\geq 1$ and let $N=\binom{n+d}{d} - n-1$.  For any $i$ we have
\[
 \binom{N}{i}\cdot N^{-n} \leq \beta_i(\PP^n; d)  \leq \binom{N}{i} \cdot N^n.
\]
\end{thm}

When $d$ is small, the bounds are not terribly interesting.
For instance, for the $5$-uple embedding $\PP^2\subseteq \PP^{20}$, let us use this to estimate $\beta_7(\PP^2; 5)$.  We have
\[
\binom{18}{7}\cdot 18^{-2} = 98 \tfrac{2}{9} \leq \beta_7(\PP^2; 5)\leq 10310976 = \binom{18}{7}\cdot 18^2.
\]
The actual Beti number is $\beta_7(\PP^2; 5)=417690$, and so the bounds are correct if unimpressive.
However, as $d$ grows, the binomial term $\binom{N}{i}$ will totally overwhelm the error term of $N^{\pm n}$.
To demonstrate this, let us consider an example where $d$ is quite large.
\begin{example}\label{ex:P2million}
Consider the embedding of $\PP^2$  by the degree $10^6$ Veronese.  In this case $N\approx 5 \times 10^{11}$.  
Using some fairly basic estimation tricks (see \S\ref{sec:examples}) our theorem implies that
\[
10^{108661150967}\leq  \beta_{10^{11}}(\PP^2; 10^6)\leq 10^{108661151026}.
\]
Thus $\beta_{10^{11}}(\PP^2; 10^6)$ has approximately $108661150996$ digits, with an error of $\pm 30$ digits.\footnote{Since $N^{\pm 2} \approx 10^{\pm 22}$, our estimate could be tightened with a better approximation of the binomial term.}
\end{example}

We obtain similar results for other varieties.  Let $X$ be a variety and let $L$ be a very ample line bundle on $X$.  We let $|L|$ be the projectived vector space $H^0(X,L)$ so that $\dim |L| = \dim H^0(X,L)-1$.  Consider the embedding $X\subseteq \PP^{\dim |L|}$.  Let $S$ be the coordinate ring of $\PP^{\dim |L|}$ and let $I\subseteq S$ be the defining ideal of $X$.  Write $\beta_i(X;L):= \dim \Tor_i(S/I,k)$.
\begin{thm}\label{thm:varieties}
Let $X$ be a variety and $L\in \Pic(X)$ and continue with notation as in the previous paragraph.  Let $r$ be the Castelnuovo-Mumford regularity of $S/I$.
For any $i$ we have:
\[
\binom{\dim |L| -\dim X}{i}(\dim |L|)^{-r} \leq \beta_i(X;L) \leq \binom{\dim |L|}{i}(\dim |L|)^{r}.
\]
\end{thm}
As above, this result is most dramatic when $L$ is very positive.
\begin{example}\label{ex:hypersurface}
Let $X\subseteq \PP^3$ be a degree $13$ hypersurface and let $L$ be the pullback of $\cO_{\PP^3}(1000)$ to $X$.  We have $\dim |L|=6441720\approx 6.4\times 10^{6}$.  
We obtain the estimate
\[
10^{1207666} \leq  \beta_{10^{6}}(X; L) \leq 10^{1207714}.
\]
Thus, this Betti number has about $1207690$ digits, with an error of $\pm 24$ digits.  Again, see \S\ref{sec:examples} for details on this computation.
\end{example}

The above results are both corollaries of our main result, which is a pair of more general algebraic bounds.
\begin{thm}\label{thm:algebraic}
Let $M$ be a finitely generated, graded $S=k[x_1, \dots, x_n]$-module, generated in degree $0$. We write $\beta_i(M):=\dim_k \Tor_i(M,k)$.  We have\footnote{If $\codim M=0$ then the lower bound should be understood as $1$.}
\[
\binom{\codim M}{i}(\codim M)^{-\reg M} \leq \frac{\beta_i(M)}{\beta_0(M)} \leq \binom{\pdim M}{i}(\pdim M)^{\reg M}.
\]
\end{thm}

Here is the idea for the proof of Theorem~\ref{thm:algebraic}. The Betti table of $M$ will have $\pdim(M)+1$ columns and $\reg(M)+1$ rows.  Boij-S\"oderberg theory provides a decomposition of the Betti table of $M$ as a sum of basic building blocks called \defi{pure diagrams}, 
where each {pure diagram} is a particular Betti table with one nonzero entry in each column.  For instance, here are potential shapes of several pure diagrams, each with $8$ columns and $3$ rows; each dot represents a nonzero entry.
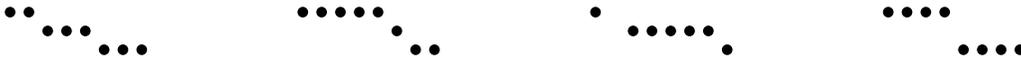
\[
\begin{tikzpicture}
 \foreach \i in {1,...,2} {
   \node[fill,circle,inner sep = 0.05cm] at (\i*0.25cm,0) {};
 }
 \foreach \i in {3,...,5} {
   \node[fill,circle,inner sep = 0.05cm] at (\i*0.25cm,-.25cm) {};
 }
  \foreach \i in {6,...,8} {
   \node[fill,circle,inner sep = 0.05cm] at (\i*0.25cm,-.5cm) {};
 }
\end{tikzpicture}
\hspace{2cm}
\begin{tikzpicture}
 \foreach \i in {1,...,5} {
   \node[fill,circle,inner sep = 0.05cm] at (\i*0.25cm,0) {};
 }
 \foreach \i in {6,...,6} {
   \node[fill,circle,inner sep = 0.05cm] at (\i*0.25cm,-.25cm) {};
 }
  \foreach \i in {7,...,8} {
   \node[fill,circle,inner sep = 0.05cm] at (\i*0.25cm,-.5cm) {};
 }
\end{tikzpicture}
\hspace{2cm}
\begin{tikzpicture}
 \foreach \i in {1,...,1} {
   \node[fill,circle,inner sep = 0.05cm] at (\i*0.25cm,0) {};
 }
 \foreach \i in {3,...,7} {
   \node[fill,circle,inner sep = 0.05cm] at (\i*0.25cm,-.25cm) {};
 }
  \foreach \i in {8,...,8} {
   \node[fill,circle,inner sep = 0.05cm] at (\i*0.25cm,-.5cm) {};
 }
\end{tikzpicture}
\hspace{2cm}
\begin{tikzpicture}
 \foreach \i in {1,...,4} {
   \node[fill,circle,inner sep = 0.05cm] at (\i*0.25cm,0) {};
 }
  \foreach \i in {5,...,8} {
   \node[fill,circle,inner sep = 0.05cm] at (\i*0.25cm,-.5cm) {};
 }
\end{tikzpicture}
\]
The main result of Boij-S\"oderberg theory (Theorem~\ref{thm:mainBS} below) decomposes $\beta(M)$ as a sum of pure diagrams.  For the pure diagrams appearing in this decomposition, the number of columns will range between $\codim(M)+1$ and $\pdim(M)+1$, and the number of rows will be at most $\reg(M)+1$. The theory was conjectured by Boij and S\"oderberg in~\cite{boij-sod1}, and proven by Eisenbud and Schreyer in \cite{eis-schrey1}, in combination with~\cite{efw,boij-sod2} and more.  Properties of $\beta(M)$ can be extracted by analyzing the corresponding pure diagrams.

Imagine now that $\codim(M) \gg \reg(M)$; each of the pure diagrams will then have many more columns than rows (see Figure~\ref{fig:twodiagrams}).
\begin{figure}
\begin{tikzpicture}
 \foreach \i in {1,...,301} {
   \node[fill,circle,inner sep = 0.01cm] at (\i*0.05cm,0) {};
 }
\end{tikzpicture}

\vspace{.5cm}

\begin{tikzpicture}
 \foreach \i in {1,...,43} {
   \node[fill,circle,inner sep = 0.01cm] at (\i*0.05cm,0) {};
 }
 \foreach \i in {44,...,211} {
   \node[fill,circle,inner sep = 0.01cm] at (\i*0.05cm,-.05cm) {};
   }

 \foreach \i in {212,...,301} {
   \node[fill,circle,inner sep = 0.01cm] at (\i*0.05cm,-.1cm) {};
 }
\end{tikzpicture}

\begin{tikzpicture}
 \foreach \i in {1} {
   \node[fill,circle,inner sep = 0.01cm] at (\i*0.5cm,0) {};
 }
\end{tikzpicture}
\caption{The heuristic behind our results is that the numerical properties of a pure diagram shaped like the bottom figure (300 columns and 3 rows) will be approximately the same as for the top figure (300 columns and 1 row).  
}
\label{fig:twodiagrams}
\end{figure}
A pure diagram with only $1$ row will correspond to a Koszul complex, and thus its Betti numbers precisely equal the binomial coefficients.  
Our key heuristic is: since any pure diagram with many more columns than rows will look approximately like a pure diagram with one row (e.g. the shapes in Figure~\ref{fig:twodiagrams} look similar, if you squint), any such diagram should behave approximately like a Koszul complex.

This heuristic leads to a proof of the theorems above via an analysis of the numerical properties of such pure diagrams.
We apply analytic techniques to pure diagrams, and then combine this with Boij--S\"oderberg theory to obtain our results.  The main proofs involve little beyond first-year calculus.  This basic idea has appeared before~\cite{erman-beh,mccullough,boocher-wigglesworth}, but our specific analyses are distinct from those.

\medskip

What should we make of these results?  Are they surprising or trivial?  How do they relate to other open questions about the quantitative behavior of (asymptotic) Betti numbers?  We will attempt to provide partial answers.

The lower bound in Theorem~\ref{thm:algebraic} is smaller than the bound predicted by the Buchsbaum-Eisenbud-Horrocks Conjecture (see~\cite[p.\ 453]{buchs-eis-gor} and \cite[Problem 24]{hartshorne-vb}) due to the $(\codim M)^{-\reg(M)}$ factor.  While the total version of the Buchsbaum-Eisenbud-Horrocks Conjecture is known~\cite{walker}, the conjecture on individual Betti numbers remains open, and so Theorem~\ref{thm:algebraic} provides a weak lower bound in this vein.
For the upper bound, ~\cite[Proposition~2.7]{eisenbud-gos} yields simple bounds on Betti numbers in terms of the Hilbert function, though  we do not see a clear relationship between that result and Theorem~\ref{thm:algebraic}.  

The closest result in the literature appears in the paragraph following \cite[Conjecture 3.2]{EL-progress-and-questions}, which provides an outline for how to leverage the methods of \cite{EEL-quick} to obtain a result similar to Theorem~\ref{thm:mainPn}.  That would provide bounds like those in Theorem~\ref{thm:mainPn} without any need for Boij-S\"oderberg theory, though there would be some delicacy in phrasing and proving a general result and it appears to involve a mildly larger error term.  Also, because of its reliance on \cite{EEL-quick}, that approach does not appear to provide an analogue of the more general Theorem~\ref{thm:algebraic}.
But we stress that, while we are unaware of a simpler argument for bounds like those in Theorem~\ref{thm:algebraic}, we would not be surprised to discover one, especially for the upper bound.

The gap between the lower and upper bounds in all of these results can be huge, but it nevertheless seems remarkable that we can estimate the order of magnitude of algebraic Betti numbers with even this much accuracy. 
 In fact, we feel that this paper primarily illustrates the curious power of Boij-S\"oderberg theory, and how certain repercussions of that theory remain unintuitive.

 Our results also add somewhat to the literature on asymptotic syzygies.  In addition to the references cited above, this robust literature includes variants of Ein and Lazarsfeld's initial asymptotic results~\cite{bruce,CJW, EEL-quick, martinova, park-nonvanishing, raicu, zhou}, computational and experimental work~\cite{wouter,ppcomputations, bruce-erman-goldstein-yang}, probability-based models ~\cite{banerjee, booms-erman-yang, engstrom,dochtermann,erman-yang-flag} and more.
Theorems \ref{thm:mainPn} and \ref{thm:varieties} are in the spirt of \cite[Conjecture D]{EEL-random}, which proposes that asymptotically, each row of the Betti table will converge to a normal distribution in a certain sense.  However, our results
only address total Betti numbers, and this is a major simplification as it entirely avoids the interactions between Betti numbers in different rows, which is a key subtlety.

This paper is structured as follows.  In \S\ref{sec:BStheory} we review the Boij-S\"oderberg theory that we will use, and in \S\ref{sec:KeyLemma} we perform our key analysis of pure diagrams.  In \S\ref{sec:proofs} we prove the main results and consider some examples.

\section*{Acknowledgments}  This paper is an outgrowth of conversations with Lawrence Ein and Rob Lazarsfeld, and we are grateful for their generosity and their comments.  We also thank Christine Berkesch, David Eisenbud, Jason McCullough, Frank-Olaf Schreyer, and Gregory G.\ Smith for useful conversations.  This paper would not have been possible without {\tt Macaulay2}~\cite{M2}.

\section{Background on Boij-S\"oderberg Theory}\label{sec:BStheory}
Boij-S\"oderberg theory began as a series of conjectures in \cite{boij-sod1}.  The initial conjectures were proven in Eisenbud and Schreyer's~\cite{eis-schrey1}, and the theory was later expanded in a number of directions, including \cite{boij-sod2,eisenbud-erman,filtering, efw,eis-schrey2, supernatural, floystad-zipping}.    See also the expository treatments~\cite{eis-schrey-icm,floystad,fmp}.  We will only summarize the small portion of Boij-S\"oderberg theory that we need for our results.

Throughout this section, $S=k[x_1,\dots, x_n]$ will denote a standard graded polynomial ring and $M$ will denote a finitely generated, graded $S$-module.  Given $d=(d_0, \dots, d_t)\in \mathbb Z^{t+1}$, we say that $d$ is a \defi{degree sequence of length $t$} if $d_{i+1}>d_i$ for all $i$.  Each degree sequence in $\mathbb Z^{t+1}$ determines a \defi{pure diagram of type $d$}, denoted $\pi_d$ where
$
 \beta_{i,j}(\pi_d)\ne 0 \iff j=d_i,
$
and where the nonzero entries of $D$ are, up to scalar multiple, given by the formula
\begin{equation}\label{eqn:herzogkuhl}
\beta_{i,d_i}(\pi_d)=\frac{\prod_{j\ne 0} d_j}{\prod_{i'\ne i} |d_i-d_{i'}|}.
\end{equation}
See for instance \cite[Definition 2.3]{boij-sod1}.
In standard Betti diagram notation, we note that the pure diagram $\pi_d$ has $t+1$ columns and $d_t-t+1$ rows. For example, if $d=(0,2,4,5)$, then \eqref{eqn:herzogkuhl} implies that
\[
\pi_d = \begin{bmatrix}
1& . & . & .\\
 . & \frac{10}{3} & . & .\\
 . & . & 5 & \frac{8}{3}\\
 \end{bmatrix}.
\]
We will use Betti number notation to refer to the entries of pure diagrams.  For instance, in the above example we will write $\beta_{1,2}(\pi_d) = \frac{10}{3}$.  
There is also a natural partial order on the degree sequences (including those of different lengths) given in ~\cite[Definition~2]{boij-sod2}, though the details of this partial order will not be very relevant for us.

The following theorem is one of the main results of Boij-S\"oderberg theory.  It was proven for Cohen--Macaulay modules in~\cite{eis-schrey1} and was extended to arbitrary modules in~\cite{boij-sod2};  see~\cite[Theorem~5.1]{floystad} or \cite[Theorem~2.1]{erman-beh} for a phrasing of the result that more closely matches the following.
\begin{thm}\label{thm:mainBS}
Let $c=\codim M$ and $p=\pdim M$.  Then there exists a unique chain of degree sequences  $d^0< \dots <d^s$ and unique positive rational numbers $c_i$ such that
	\[
	\beta(M)=\sum_{i=0}^s c_i\pi_{d^i}.
	\]
Each degree sequence $d^i$ has length at least $c$ and at most $p$.
\end{thm}
\begin{example}\label{ex:BoijSodDecomposition}
Consider $I=\langle x^2,xyz,yz^2,y^2z,z^3 \rangle$.  We have
\[
\beta(S/I) = \begin{bmatrix}
1 & . & . & .\\
  . & 1 & . & .\\
 . & 4 & 5 & 1\\
  . & . & 1 & 1
      \end{bmatrix}
\]
and the decomposition of $\beta(S/I)$ is given by
\[
\begin{footnotesize}
\frac{3}{10} 
\begin{bmatrix}
1& . & . & .\\
 . & \frac{10}{3} & . & .\\
 . & . & 5 & \frac{8}{3}\\
  . & . & . & .
 \end{bmatrix}
 +
 \frac{1}{30} 
\begin{bmatrix}
1& . & . & .\\
 . & . & . & .\\
 . & 10 & 15 & 6\\
  . & . & . & .
 \end{bmatrix}
 +
 \frac{1}{3} 
\begin{bmatrix}
1& . & . & .\\
 . & . & .  & .\\
 . & 8 & 9 & .\\
  . & . & . & 2
 \end{bmatrix}
 +
  \frac{1}{15} 
\begin{bmatrix}
1& . & . & .\\
 . & . & .  & .\\
 . & 5 & . & .\\
  . & . & 9 & 5
 \end{bmatrix}
 +
   \frac{4}{15} 
\begin{bmatrix}
1& . & . & .\\
 . & . & .  & .\\
 . & \frac{5}{2} & . & .\\
  . & . & \frac{3}{2} & .
 \end{bmatrix}.
      \end{footnotesize}
\]
From left to right, these are pure diagrams of type $(0,2,4,5), (0,3,4,5), (0,3,5,6)$ and $(0,3,5)$.  Note that, because we have normalized the pure diagrams so that $\beta_0$ is always $1$, the sum of the coefficients also equals $1$.
\end{example}

\section{Numerics of pure diagrams}\label{sec:KeyLemma}
Our main lemma is the following result on the numerical properties of pure diagrams.
\begin{lemma}\label{keyLemma}
Fix $N\geq 1,r\geq 0$.  Let $d=(d_0, d_1, \dots, d_N)$ be a degree sequence with $d_0=0$ and $d_N \leq N+r$.   Let $\pi_d$ be the corresponding pure diagram, with entries as in \eqref{eqn:herzogkuhl}.  Then
\[
 \binom{N}{i}\cdot N^{-r} \leq \beta_i(\pi_d)  \leq \binom{N}{i} \cdot N^r.
\]
\end{lemma}

\begin{example}
Let us apply Lemma~\ref{keyLemma} to the pure diagram $\pi_d$ of type $d=(0,2,4,5)$ from Example~\ref{ex:BoijSodDecomposition}.   In the notation of the lemma, $r=2$ and $N=3$ and so the lemma would imply, for instance, that
\[
\binom{3}{3} 3^{-2} = \frac{1}{9} \leq \beta_3(\pi_d) \leq \binom{3}{3} 3^2=9.
\]
Since we have $\beta_3(\pi_d) = \frac{8}{3}$, we see that it satisfies the desired bounds.
\end{example}
\begin{example}
When $r=0$, we have $d=(0,1,\dots,N)$ and $\pi_d$ equals the Betti table of a Koszul complex.  Specifically, one can check that in this case $\beta_i(\pi_d)=\binom{N}{i}$.
\end{example}

If $d=(0,1,2,3,5)$ then
\[
\pi_d = \begin{bmatrix}
1& \frac{15}{4} & 5 & \frac{5}{2}& .\\
 . & . & . & .&\frac{1}{4}\\
 \end{bmatrix}
\]
and the lower bound is sharp for $\beta_4$.  Similarly if $d=(0,2,3,4,5)$ then the upper bound will be sharp for $\beta_4$.  In general, we do not expect the error bound of $N^{\pm r}$ to be sharp for all $N,r$.  But if we consider $d=(0,1+r,2+r, \dots, N+r)$  then $\beta_N = \frac{(1+r)(2+r)\cdots (N+r)}{(N+r)(N-1)(N-2)\cdots 1} = \binom{N+r-1}{r} \approx \binom{N}{N} \cdot O(N^r)$, and so every $N,r$ there is some pure diagram and Betti number where the upper bound has the right order of magnitude.  Similarly if we consider $\beta_N$ for $d=(0,1,2,\dots,N-1,N+r)$, then we see that the lower bound has the right order of magnitude.

\begin{proof}[Proof of Lemma~\ref{keyLemma}]
Since $d_0= 0, d_N \leq N+r$ and each $d_{i+1} \geq d_i+1$, we have $j \leq d_j \leq d_j+r$ for all $j$.  As we are interested in the $i$'th Betti number, we will write $d_i = i+a$ for some $0\leq a \leq r$.    This adds the further restriction that $j\leq d_j \leq d_j+a$ for $0<j<i$ and that $j+a \leq d_j \leq d_j+r$ for $i<j\leq N$.

Using the formula from \eqref{eqn:herzogkuhl}, we now have
\[
\beta_i(\pi_d) = \frac{\prod_{j=1}^N d_j}{\prod_{j=0}^{i-1} (i+a-d_j)  \cdot \prod_{j=i+1}^{N} (d_j-i-a)}. 
\]
The main idea in this proof (as well as in previous papers like~\cite{erman-beh,mccullough,boocher-wigglesworth}) is to consider this as a rational function in the variables $d_j$ for $j\ne i$, and to use optimization techniques to bound the potential values.  For $0<j<i$ we have $d_j \in [j,j+a]$ and for $i<j<N$ we have $d_j \in [j+a, j+r]$.
We want to determine whether the function is increasing or decreasing in $d_j$ for each $j$.  

Here is one key trick: a positive function $f(x)$ is increasing if and only if $\log f(x)$ is increasing.  We can thus apply $\log$ to the functional expression for $ \beta_i(\pi_d)$, transforming the products into sums, and making it easier to analyze the partial derivatives.  Namely, since
\[
\log \beta_i(\pi_d) = \left( \sum_{j=1}^N \log d_j\right)  -\left( \sum_{ j=0}^{i-1} \log (i+a-d_j) \right) - \left( \sum_{j=i+1}^N \log (d_j-i-a) \right)
\]
we can compute:
\[
\frac{\partial}{\partial d_j} \log \beta_i(\pi_d) = 
\begin{cases}
\frac{1}{d_j} + \frac{1}{i+a-d_j} = \frac{i+a}{d_j(i+a-d_j)} & \text{ if } j < i\\
\frac{1}{d_j}  - \frac{1}{d_j-i-a} = \frac{-i-a}{d_j(d_j-i-a)}& \text{ if } j > i.
\end{cases}
\]
Note that for $0<j<i$ we have $\frac{i+a}{d_j(i+a-d_j)}>0$ for all $d_j \in [j,j+a]$ and thus these functions are always increasing in $d_j$ for $0<j<i$.  Conversely, for $i<j\leq N$ we have $ \frac{-i-a}{d_j(d_j-i-a)}<0$ for all $d_j \in [j+a, j+r]$,and thus  these functions are always decreasing in $d_j$ for $i<j\leq N$.  It follows that the maximal and minimal values are achieved at boundary points.  The maximal value is achieved when $d_j = j+a$ for both $j<i$ and $j>i$.  Thus the maximal value occurs for $d_{\max} = (0,a+1, \dots, N+a)$.  The minimal value occurs at the opposite extremes, yielding $d_{\min}=(0,1,\dots,i-1,i+a,i+r+1,i+r+2, \dots, N+r)$.

To streamline notation in the following computations, we will write $a \cdots b$ for the descending product $a\cdot (a-1)\cdot (a-2)\cdots b$.

{\bf Upper bound:} We have
\begin{align*}
\beta_i(\pi_{d_{\max}})&=
 \prod_{j=1}^N d_j \cdot \frac{1}{\prod_{j=0}^{i-1} (i+a-d_j) } \cdot  \frac{1}{\prod_{j=i+1}^{N} (d_j-i-a)}\\
 &= \tfrac{(N+a)\cdots (a+1)}{1} \cdot \tfrac{1}{(i+a) \cdot (i-1)\cdots 1} \cdot \tfrac{1}{(N-i) \cdots 1}\\
 &=\tfrac{(N+a)!}{a!} \cdot  \frac{1}{(i+a)(i-1)!} \cdot \frac{1}{(N-i)!}\\
 &= \binom{N}{i} \cdot \frac{(N+a)\cdots (N+1)}{a!} \cdot \frac{i}{i+a}\\
 \intertext{Now we apply Lemma~\ref{lem:maxLemma} to obtain $\frac{(N+a)\cdots (N+1)}{a!} \cdot \frac{i}{i+a}\leq N^a$ and since $a\leq r$ we get}
 &\beta_i(\pi_{d_{\max}})\leq \binom{N}{i} \cdot N^a \leq \binom{N}{i}\cdot N^r.
\end{align*}
It follows that for any degree sequence $d$ as in the statement of the lemma, we have
\[
\beta_i(\pi_d) \leq \beta_i(\pi_{d_{\max}}) \leq \binom{N}{i}\cdot N^r.
\]

{\bf Lower bound:}  For the lower bound, recall that we have $d_{\min}:=(0,1,\dots,i-1,i+a,i+r+1,i+r+2, \dots, N+r)$ and so
we have
\begin{align*}
\beta_i(\pi_{d_{\min}}) &=
 \prod_{j=1}^N d_j \cdot \frac{1}{\prod_{j=0}^{i-1} (i+a-d_j) } \cdot  \frac{1}{\prod_{j=i+1}^{N} (d_j-i-a)}\\
&=\tfrac{(N+r) \cdots (i+r+1) \cdot (i+a) \cdot (i-1)!}{1} \cdot \tfrac{1}{(i+a) \cdots (a+1)} \cdot\tfrac{1}{(N+r-i-a)\cdots (r-a+1)}\\
&=\tfrac{(N+r) \cdots (N+1) \cdot N!(i+a)}{(i+r)\cdots i} \cdot \tfrac{a!}{i!(i+a)\cdots (i+1)} \cdot\tfrac{(r-a)!}{ (N-i)! \cdot (N-i+r-a) \cdots (N-i+1)}\\
&=\binom{N}{i}\tfrac{(N+r) \cdots (N+1)\cdot (i+a)}{(i+r)\cdots i} \cdot \tfrac{a!}{(i+a)\cdots (i+1)} \cdot\tfrac{(r-a)!}{ (N-i+r-a) \cdots (N-i+1)}\\
&=\binom{N}{i} \cdot \tfrac{(N+r) \cdots (N+1)}{(i+r)\cdots (i+1)}\cdot \left(\tfrac{i+a}{i} \tfrac{a!}{(i+a)\cdots (i+1)}\right) \cdot\left(\tfrac{(r-a)!}{ (N-i+r-a) \cdots (N-i+1)}\right)\\
\intertext{By Lemma~\ref{lem:maxLemma}, the third term $ \left(\tfrac{i+a}{i} \tfrac{a!}{(i+a)\cdots (i+1)}\right)$ is at least $i^{-a}$.  By Lemma~\ref{lem:descending}, the fourth term $\left(\tfrac{(r-a)!}{ (N-i+r-a) \cdots (N-i+1)}\right)$ is at least $(N-i+1)^{-(r-a)}$.  The second term $\tfrac{(N+r) \cdots (N+1)}{(i+r)\cdots (i+1)}$ can be written as $\prod_{j=1}^r \frac{N+j}{i+j}$ which is at least $1$.  Combining these we get:}
\beta_i(\pi_{d_{\min}}) &\geq \binom{N}{i} \cdot 1 \cdot i^{-a} \cdot (N-i+1)^{r-a}
\intertext{Since $1\leq i \leq N$ we have $i^{-a} \geq N^{-a}$ and $(N-i+1)^{-(r-a)}\geq N^{-(r-a)}$ yielding}
\beta_i(\pi_{d_{\min}}) &\geq \binom{N}{i} \cdot N^{-a} \cdot N^{-(r-a)} =\binom{N}{i} N^{-r}.
\end{align*}
Since $\beta(\pi_d) \geq \beta(\pi_{d_{\min}})$, this completes the proof of the lower bound.
\end{proof}

\begin{lemma}\label{lem:maxLemma}
Fix nonnegative integers $N,i,a$ where $1 \leq N$ and $1\leq i\leq N$.  We have
\[
 \frac{(N+a)\cdots (N+1)}{a!}\frac{i}{i+a} \leq N^a.
 \]
\end{lemma}
\begin{proof}
If $a=0$ we have $1=1$.  
Now we will induct on $a$.  We have
\begin{align*}
 \frac{(N+a)\cdots (N+1)}{a!}\frac{i}{i+a} &=  \left( \frac{(N+a-1)\cdots (N+1)}{(a-1)!}\frac{i}{i+a-1} \right) \cdot \frac{N+a}{a} \cdot \frac{i+a-1}{i+a}.
 \intertext{Applying the induction hypothesis to the parenthetical term we get}
 &\leq  \left( N^{a-1} \right) \cdot \frac{N+a}{a} \cdot \frac{i+a-1}{i+a}.
\end{align*}
So we have reduced to proving that $\frac{N+a}{a} \cdot \frac{i+a-1}{i+a} \leq N$.  The term $ \frac{i+a-1}{i+a}$ is largest when $i=N$.  Thus we now have
\[
\frac{N+a}{a} \cdot \frac{i+a-1}{i+a} \leq  \frac{N+a}{a}\cdot \frac{N+a-1}{N+a} = \frac{N+a-1}{a}.
\]
So we have further reduced to proving $\frac{N+a-1}{a}\leq N$ for all $a\geq 1$.  This follows as
\[
\frac{N+a-1}{a}\leq N \iff 0 \leq Na - (N+a-1) = (N-1)(a-1)
\]
which holds because $a,N\geq 1$.
\end{proof}

\begin{lemma}\label{lem:descending}
For $a\geq 0$ and $b\geq 0$, we have $\frac{(a+b)\cdots (a+1)}{b!} \leq (a+1)^b$.
\end{lemma}
\begin{proof}
We induct on $b$.  When $b=0$ the statement becomes $1=1$.  More generally we have $\frac{(a+b-1)\cdots (a+1)}{(b-1)!} \leq (a+1)^{b-1}$ by induction.  We claim also that $\frac{a+b}{b}\leq a+1$ for $b\geq 1$ and $a\geq 0$.  This is because
\[
\frac{a+b}{b} \leq a+1 \iff 0 \leq (a+1)b - (a+b) = a(b-1)
\]
which holds when $b\geq 1$ and $a\geq 0$.
Multiplying $\frac{(a+b-1)\cdots (a+1)}{(b-1)!} \leq (a+1)^{b-1}$ and $\frac{a+b}{b}\leq a+1$ gives the desired inequality for $a\geq 0$ and $b \geq 1$, completing the induction step.
\end{proof}

\section{Proofs of Theorems}\label{sec:proofs}
Let us begin with a proof of the main algebraic result.

\begin{proof}[Proof of Theorem~\ref{thm:algebraic}]
Let $\beta(M)=\sum_{i=0}^s c_i\pi_{d^i}$ be the Boij-S\"oderberg decomposition, so that $c_i \in \QQ_{>0}$ and $d^0<d^1<\cdots <d^s$ forms a chain of degree sequences.  Let us rescale $\frac{1}{\beta_0(M)} \cdot \beta(M) = \sum_{i=1}^s \frac{c_i}{\beta_0(M)} \pi_{d^i}$.  Since $\beta_0( \pi_{d^i})=1$ for all pure diagrams, by the definition \eqref{eqn:herzogkuhl}, it follows that $\sum_{i=1}^s \frac{c_i}{\beta_0(M)} = 1$.  Each $d^i$ is a degree sequence of length $t$ where $\codim(M) \leq t \leq \pdim(M)$ by Theorem~\ref{thm:mainBS}.  Moreover, each $\pi_{d^i}$ will involve at most $\reg(M)$ rows, and thus $(d^i)_t\leq t+\reg(M)$.  Lemma~\ref{keyLemma} implies that
\[
\binom{t}{i}t^{-\reg M} \leq \beta_i(\pi_{d^i}) \leq \binom{t}{i}t^{\reg M}.
\]
Incorporating that $\codim M \leq t\leq \pdim M$ we get
\[
\binom{\codim M}{i}(\codim M)^{-\reg M} \leq \beta_i(\pi_{d^i}) \leq \binom{\pdim M}{i}(\pdim M)^{\reg M}.
\]
Since this holds for each pure diagram, and since $\sum_{i=1}^s \frac{c_i}{\beta_0(M)} = 1$, we conclude that
\[
\binom{\codim M}{i}(\codim M)^{-\reg M} \leq \frac{\beta_i(M)}{\beta_0(M)} \leq \binom{\pdim M}{i}(\pdim M)^{\reg M}.
\]
\end{proof}

Before proving Theorem~\ref{thm:mainPn} we will prove the following lemma. The statement is well-known to experts, but we provide an independent proof for the sake of the reader.

\begin{lemma}\label{lem:reg}
With hypotheses as in Theorem~\ref{thm:mainPn}, the Castelnuovo-Mumford regularity of $S/I$ is at most $n$.
\end{lemma}
\begin{proof}
Let $\mathfrak m$ be the homogeneous maximal ideal of $S$.  For an integer $m$, we have by~\cite[Theorem 4.3]{eisenbud-gos} that $m\geq \reg(S/I)$ if and only if
\[
m\geq \max \{ e  | H^{j}_{\mathfrak m}(S/I)_e \ne 0\} + j \text{ for all } i\geq 0.
\]
Since $S/I$ is a Cohen--Macaulay ring of dimension $n+1$, we have $H^j_{\mathfrak m}(S/I)=0$ for all $j\ne n+1$.  Moreover, the pullback of $\cO_{\PP^{N+n}}(e))$ to $\PP^n$ is $\cO_{\PP^n}(de))$.  We then have $H^{n+1}_{\mathfrak m}(S/I)_e = H^{n}(\PP^n, \cO_{\PP^n}(ed))$ which is nonzero if and only if $ed \leq -n-1$ if and only if $e \leq \lfloor \frac{-n-1}{d}\rfloor$.  It follows that $m$ satisfies the condition of the above displayed equation if and only if $m\geq \lfloor \frac{-n-1}{d}\rfloor+(n+1)$.    Since $\lfloor \frac{-n-1}{d}\rfloor$ is at most $-1$, we conclude that $m=n$ satisfies the desired condition, and thus $\reg(S/I) \leq n$.
\end{proof}

\begin{proof}[Proof of Theorem~\ref{thm:mainPn}]
We will apply Theorem~\ref{thm:algebraic} to $S/I$.  Since $S/I$ is Cohen-Macaulay, we have $\codim S/I = \pdim S/I = N$.  By Lemma~\ref{lem:reg} we have $\reg(S/I) \leq n$.  We thus have
\[
\binom{N}{i} N^{-n} \leq \binom{N}{i} N^{-\reg S/I} \leq \beta_i(S/I) = \beta_i(\PP^n;d) \leq \binom{N}{i} N^{\reg S/I}  \leq  \binom{N}{i} N^n.
\]
\end{proof}
\begin{proof}[Proof of Theorem~\ref{thm:varieties}]
We will apply Theorem~\ref{thm:algebraic} to $S/I$.  We have $X\subseteq \PP^{\dim |L|}$ and thus $S$ is a polynomial ring with $\dim |L|+1$ variables.  We have $\codim S/I = \dim |L| - \dim X$.  Since $S/I$ has depth at least $1$, the Auslander--Buchsbaum Theorem implies that $\pdim S/I \leq \dim S - 1 = \dim |L|.$   The statement then follows immediately from Theorem~\ref{thm:algebraic}.
\end{proof}

\section{Examples}\label{sec:examples}
In this final subsection, we will provide some details on how we made estimates for our various examples.  Our techniques are all well-known and non-optimal, but we include them for clarity.  Imagine that we want to estimate
\[
\log(a! / b!)  =\log(a) + \log(a-1) + \cdots + \log(b+1).
\]
Since $\log$ is an increasing function, we have 
\[
\int_{b}^{a} \log(x) \ dx \leq \log(a) + \log(a-1) + \cdots + \log(b+1) \leq \int_{b+1}^{a+1} \log (x) \ dx.
\]
Since $\int \log(x) \ dx= x\log(x) - x+C$ we obtain
\[
a\log(a) - a - b\log(b) + b  \leq \log(a! / b!) \leq (a+1)\log(a+1) - (a+1) -(b+1)\log(b+1) + (b+1).
\]
For $c>1$, we also have the simpler bounds
\[
c\log(c) - c \leq \log(c!) \leq (c+1)\log(c+1) - (c+1).
\]
Putting this together, with $a=N, b=N-i$ and $c=i$, we get the lower bound
\begin{equation}\label{eqn:logBinomialLower}
N\log(N) -  (N-i)\log(N-i) - (i+1)\log(i+1) +1\leq \log \binom{N}{i}.
\end{equation}
The upper bound is similar.

Let us use this to perform the estimates in Example~\ref{ex:P2million}.  In that case we have
\[
N=\binom{10^6+2}{2} -3 =\frac{1000000 \cdot 999999}{2} - 3 = 500001499998.
\]  
We have chosen $i=10^{11}$ and so want to bound $\binom{N}{10^{11}}$.  Applying \eqref{eqn:logBinomialLower} in this case we get $\log \binom{N}{i} \approx 250201546457.083$.  Converting from natural log to log base $10$ we get $\log_{10} \binom{N}{i} \approx 108661150989.971$.  Incorporating the error of $\log_{10} N^{-2}=-2\log_{10} N$, and rounding down to get an integer, we end up with a lower bound of $108661150967$.  The upper bound computation is similar.

For the Example~\ref{ex:hypersurface}, we have an exact sequence
\[
0\to \cO_{\PP^3}(1000-13) \to \cO_{\PP^3}(1000)\to L\to 0.
\]
We thus have
\[
\dim |L| = \binom{1000+3}{3} - \binom{1000-13 + 3}{3} - 1 = 6441720.
\]
We have chosen $i=10^{6}$ in this case.  By a computation similar to Lemma~\ref{lem:reg}, the coordinate ring of $S/I$ in this example will have regularity $\dim X + 1 =3$.
Applying \eqref{eqn:logBinomialLower} we get $\binom{\dim |L|-\dim X}{i} = \binom{6441720 - 2}{10^6}\approx 10^{1207690}$ and $(\dim |L|)^{-r}\approx 10^{-24}$, yielding the lower bound from that example.  The upper bound estimation is similar.


%


\begin{bibdiv}
\begin{biblist}

\bib{banerjee}{article}{
   author={Banerjee, Arindam},
   author={Yogeshwaran, D.},
   title={Edge ideals of Erd\H os-R\'enyi random graphs: linear resolution,
   unmixedness and regularity},
   journal={J. Algebraic Combin.},
   volume={58},
   date={2023},
   number={4},
   pages={1125--1154},
}
\bib{boij-sod1}{article}{
   author={Boij, Mats},
   author={S\"oderberg, Jonas},
   title={Betti numbers of graded modules and the multiplicity conjecture in
   the non-Cohen-Macaulay case},
   journal={Algebra Number Theory},
   volume={6},
   date={2012},
   number={3},
   pages={437--454},
}

\bib{boij-sod2}{article}{
   author={Boij, Mats},
   author={S\"oderberg, Jonas},
   title={Graded Betti numbers of Cohen-Macaulay modules and the
   multiplicity conjecture},
   journal={J. Lond. Math. Soc. (2)},
   volume={78},
   date={2008},
   number={1},
   pages={85--106},
}

\bib{boocher-wigglesworth}{article}{
   author={Boocher, Adam},
   author={Wigglesworth, Derrick},
   title={Large lower bounds for the betti numbers of graded modules with
   low regularity},
   journal={Collect. Math.},
   volume={72},
   date={2021},
   number={2},
   pages={393--410},
}

\bib{booms-erman-yang}{article}{
   author={Booms-Peot, Caitlyn},
   author={Erman, Daniel},
   author={Yang, Jay},
   title={Characteristic dependence of syzygies of random monomial ideals},
   journal={SIAM J. Discrete Math.},
   volume={36},
   date={2022},
   number={1},
   pages={682--701},
}


\bib{bruce}{article}{
   author={Bruce, Juliette},
   title={The quantitative behavior of asymptotic syzygies for Hirzebruch
   surfaces},
   journal={J. Commut. Algebra},
   volume={14},
   date={2022},
   number={1},
   pages={19--26},
}


\bib{ppcomputations}{article}{
   author={Bruce, Juliette},
   author={Corey, Daniel},
   author={Erman, Daniel},
   author={Goldstein, Steve},
   author={Laudone, Robert P.},
   author={Yang, Jay},
   title={Syzygies of $\Bbb P^1\times\Bbb P^1$: data and conjectures},
   journal={J. Algebra},
   volume={593},
   date={2022},
   pages={589--621},
}

\bib{bruce-erman-goldstein-yang}{article}{
   author={Bruce, Juliette},
   author={Erman, Daniel},
   author={Goldstein, Steve},
   author={Yang, Jay},
   title={Conjectures and computations about Veronese syzygies},
   journal={Exp. Math.},
   volume={29},
   date={2020},
   number={4},
   pages={398--413},
}



\bib{buchs-eis-gor}{article}{
    AUTHOR = {Buchsbaum, David A.},
    AUTHOR = {Eisenbud, David},
     TITLE = {Algebra structures for finite free resolutions, and some
              structure theorems for ideals of codimension {$3$}},
   JOURNAL = {Amer. J. Math.},
  FJOURNAL = {American Journal of Mathematics},
    VOLUME = {99},
      YEAR = {1977},
    NUMBER = {3},
     PAGES = {447--485},
}


\bib{wouter}{article}{
   author={Castryck, Wouter},
   author={Cools, Filip},
   author={Demeyer, Jeroen},
   author={Lemmens, Alexander},
   title={Computing graded Betti tables of toric surfaces},
   journal={Trans. Amer. Math. Soc.},
   volume={372},
   date={2019},
   number={10},
   pages={6869--6903},
}

\bib{CJW}{article}{
   author={Conca, Aldo},
   author={Juhnke-Kubitzke, Martina},
   author={Welker, Volkmar},
   title={Asymptotic syzygies of Stanley-Reisner rings of iterated
   subdivisions},
   journal={Trans. Amer. Math. Soc.},
   volume={370},
   date={2018},
   number={3},
   pages={1661--1691},
   }




\bib{dochtermann}{article}{
   author={Dochtermann, Anton},
   author={Newman, Andrew},
   title={Random subcomplexes and Betti numbers of random edge ideals},
   journal={Int. Math. Res. Not. IMRN},
   date={2023},
   number={10},
   pages={8832--8871},
}


\bib{EEL-quick}{article}{
   author={Ein, Lawrence},
   author={Erman, Daniel},
   author={Lazarsfeld, Robert},
   title={A quick proof of nonvanishing for asymptotic syzygies},
   journal={Algebr. Geom.},
   volume={3},
   date={2016},
   number={2},
   pages={211--222},
}

\bib{EEL-random}{article}{
   author={Ein, Lawrence},
   author={Erman, Daniel},
   author={Lazarsfeld, Robert},
   title={Asymptotics of random Betti tables},
   journal={J. Reine Angew. Math.},
   volume={702},
   date={2015},
   pages={55--75},
}
\bib{EL-higher}{article}{
   author={Ein, Lawrence},
   author={Lazarsfeld, Robert},
   title={Syzygies and Koszul cohomology of smooth projective varieties of
   arbitrary dimension},
   journal={Invent. Math.},
   volume={111},
   date={1993},
   number={1},
   pages={51--67},}


\bib{EL-asymptotic}{article}{
   author={Ein, Lawrence},
   author={Lazarsfeld, Robert},
   title={Asymptotic syzygies of algebraic varieties},
   journal={Invent. Math.},
   volume={190},
   date={2012},
   number={3},
   pages={603--646},
}

\bib{EL-progress-and-questions}{article}{
   author={Ein, Lawrence},
   author={Lazarsfeld, Robert},
   title={Syzygies of projective varieties of large degree: recent progress
   and open problems},
   conference={
      title={Algebraic geometry: Salt Lake City 2015},
   },
   book={
      series={Proc. Sympos. Pure Math.},
      volume={97.1},
      publisher={Amer. Math. Soc., Providence, RI},
   },
   isbn={978-1-4704-3577-6},
   date={2018},
   pages={223--242},
}

\bib{eisenbud-gos}{book}{
   author={Eisenbud, David},
   title={The geometry of syzygies},
   series={Graduate Texts in Mathematics},
   volume={229},
   note={A second course in commutative algebra and algebraic geometry},
   publisher={Springer-Verlag, New York},
   date={2005},
   pages={xvi+243},
}

\bib{eisenbud-erman}{article}{
   author={Eisenbud, David},
   author={Erman, Daniel},
   title={Categorified duality in Boij-S\"oderberg theory and invariants of
   free complexes},
   journal={J. Eur. Math. Soc. (JEMS)},
   volume={19},
   date={2017},
   number={9},
   pages={2657--2695},
}


\bib{filtering}{article}{
   author={Eisenbud, David},
   author={Erman, Daniel},
   author={Schreyer, Frank-Olaf},
   title={Filtering free resolutions},
   journal={Compos. Math.},
   volume={149},
   date={2013},
   number={5},
   pages={754--772},
}

\bib{efw}{article}{
   author={Eisenbud, David},
   author={Fl\o ystad, Gunnar},
   author={Weyman, Jerzy},
   title={The existence of equivariant pure free resolutions},
   language={English, with English and French summaries},
   journal={Ann. Inst. Fourier (Grenoble)},
   volume={61},
   date={2011},
   number={3},
   pages={905--926},
}

\bib{eis-schrey1}{article}{
   author={Eisenbud, David},
   author={Schreyer, Frank-Olaf},
   title={Betti numbers of graded modules and cohomology of vector bundles},
   journal={J. Amer. Math. Soc.},
   volume={22},
   date={2009},
   number={3},
   pages={859--888},
}

\bib{eis-schrey2}{article}{
   author={Eisenbud, David},
   author={Schreyer, Frank-Olaf},
   title={Cohomology of coherent sheaves and series of supernatural bundles},
   journal={J. Eur. Math. Soc. (JEMS)},
   volume={12},
   date={2010},
   number={3},
   pages={703--722},
}


\bib{eis-schrey-icm}{article}{
   author={Schreyer, Frank-Olaf},
   author={Eisenbud, David},
   title={Betti numbers of syzygies and cohomology of coherent sheaves},
   conference={
      title={Proceedings of the International Congress of Mathematicians.
      Volume II},
   },
   book={
      publisher={Hindustan Book Agency, New Delhi},
   },
   isbn={978-81-85931-08-3},
   isbn={978-981-4324-32-8},
   isbn={981-4324-32-9},
   date={2010},
}

\bib{engstrom}{article}{
   author={Engstr\"om, Alexander},
   author={Orlich, Milo},
   title={The regularity of almost all edge ideals},
   journal={Adv. Math.},
   volume={435},
   date={2023},
   pages={Paper No. 109355, 34},
   }


\bib{supernatural}{article}{
   author={Erman, Daniel},
   author={Sam, Steven V.},
   title={Supernatural analogues of Beilinson monads},
   journal={Compos. Math.},
   volume={152},
   date={2016},
   number={12},
   pages={2545--2562},
}

\bib{fmp}{article}{
   author={Fl\o ystad, Gunnar},
   author={McCullough, Jason},
   author={Peeva, Irena},
   title={Three themes of syzygies},
   journal={Bull. Amer. Math. Soc. (N.S.)},
   volume={53},
   date={2016},
   number={3},
   pages={415--435},
}

%
%

\bib{erman-beh}{article}{
   author={Erman, Daniel},
   title={A special case of the Buchsbaum-Eisenbud-Horrocks rank conjecture},
   journal={Math. Res. Lett.},
   volume={17},
   date={2010},
   number={6},
   pages={1079--1089},
}

\bib{erman-yang-flag}{article}{
   author={Erman, Daniel},
   author={Yang, Jay},
   title={Random flag complexes and asymptotic syzygies},
   journal={Algebra Number Theory},
   volume={12},
   date={2018},
   number={9},
   pages={2151--2166},
}



\bib{floystad}{article}{
   author={Fl\o ystad, Gunnar},
   title={Boij-S\"oderberg theory: introduction and survey},
   conference={
      title={Progress in commutative algebra 1},
   },
   book={
      publisher={de Gruyter, Berlin},
   },
   isbn={978-3-11-025034-3},
   date={2012},
   pages={1--54},
}

\bib{floystad-zipping}{article}{
   author={Fl\o ystad, Gunnar},
   title={Zipping Tate resolutions and exterior coalgebras},
   journal={J. Algebra},
   volume={437},
   date={2015},
   pages={249--307},
}


\bib{green1}{article}{
   author={Green, Mark L.},
   title={Koszul cohomology and the geometry of projective varieties. II},
   journal={J. Differential Geom.},
   volume={20},
   date={1984},
   number={1},
   pages={279--289},}


\bib{green-laz2}{article}{
   author={Green, Mark},
   author={Lazarsfeld, Robert},
   title={On the projective normality of complete linear series on an
   algebraic curve},
   journal={Invent. Math.},
   volume={83},
   date={1986},
   number={1},
   pages={73--90},
}

\bib{green-laz1}{article}{
   author={Green, M.},
   author={Lazarsfeld, R.},
   title={Some results on the syzygies of finite sets and algebraic curves},
   journal={Compositio Math.},
   volume={67},
   date={1988},
   number={3},
   pages={301--314},
}

\bib{hartshorne-vb}{article}{
    AUTHOR = {Hartshorne, Robin},
     TITLE = {Algebraic vector bundles on projective spaces: a problem list},
   JOURNAL = {Topology},
  FJOURNAL = {Topology. An International Journal of Mathematics},
    VOLUME = {18},
      YEAR = {1979},
    NUMBER = {2},
     PAGES = {117--128},
      ISSN = {0040-9383},
}


\bib{lazarsfeld-sampling}{article}{
   author={Lazarsfeld, Robert},
   title={A sampling of vector bundle techniques in the study of linear
   series},
   conference={
      title={Lectures on Riemann surfaces},
      address={Trieste},
      date={1987},
   },
   book={
      publisher={World Sci. Publ., Teaneck, NJ},
   },
   isbn={9971-50-902-4},
   date={1989},
   pages={500--559},
}


\bib{M2}{misc}{
          author = {Grayson, Daniel R.}
          author = {Stillman, Michael E.},
          title = {Macaulay2, a software system for research in algebraic geometry},
          howpublished = {Available at \url{http://www2.macaulay2.com}},
          label = {M2},
        }
        
\bib{martinova}{misc}{
   author={Martinova, Boyana},
   title={Asymptotic syzygies of weighted projective spaces},
note = {arXiv: 2510.12708},
}

\bib{mccullough}{article}{
   author={McCullough, Jason},
   title={A polynomial bound on the regularity of an ideal in terms of half
   of the syzygies},
   journal={Math. Res. Lett.},
   volume={19},
   date={2012},
   number={3},
   pages={555--565},
}


\bib{park}{article}{
   author={Park, Jinhyung},
   title={Asymptotic vanishing of syzygies of algebraic varieties},
   journal={Commun. Am. Math. Soc.},
   volume={2},
   date={2022},
   pages={133--148},
}

\bib{park-nonvanishing}{article}{
   author={Park, Jinhyung},
   title={Asymptotic nonvanishing of syzygies of algebraic varieties},
   journal={Math. Ann.},
   volume={392},
   date={2025},
   number={1},
   pages={751--779},
}

\bib{raicu}{article}{
   author={Raicu, Claudiu},
   title={Representation stability for syzygies of line bundles on
   Segre-Veronese varieties},
   journal={J. Eur. Math. Soc. (JEMS)},
   volume={18},
   date={2016},
   number={6},
   pages={1201--1231},
}
%

\bib{walker}{article}{
   author={Walker, Mark E.},
   title={Total Betti numbers of modules of finite projective dimension},
   journal={Ann. of Math. (2)},
   volume={186},
   date={2017},
   number={2},
   pages={641--646},
}

\bib{zhou}{article}{
   author={Zhou, Xin},
   title={Effective non-vanishing of asymptotic adjoint syzygies},
   journal={Proc. Amer. Math. Soc.},
   volume={142},
   date={2014},
   number={7},
   pages={2255--2264},
}


\end{biblist}
\end{bibdiv}
\end{document}